\documentclass[11pt]{amsart}
\usepackage{amsmath,amssymb, mathrsfs}
\usepackage{xypic}
\usepackage{pstricks}
\usepackage[colorlinks=true,linkcolor=black,citecolor=black,urlcolor=black]{hyperref}

\usepackage[french,english]{babel}

\psset{unit=1pt}
\psset{arrowsize=4pt 1}
\psset{linewidth=.5pt}

\newcommand{\excise}[1]{}

\newcommand{\isom}{\cong}

\renewcommand{\phi}{\varphi}

\renewcommand{\bar}{\overline}

\renewcommand{\AA}{\mathbb{A}}

\newcommand{\PP}{\mathbb{P}}

\newcommand{\lieg}{\mathfrak{g}}

\newcommand{\liet}{\mathfrak{t}}

\newcommand{\Qq}{\mathcal{Q}}       

\DeclareMathOperator{\Spec}{Spec}
\DeclareMathOperator{\Sym}{S}
\DeclareMathOperator{\GL}{GL}
\DeclareMathOperator{\ad}{ad}
\DeclareMathOperator{\Char}{char}

\newtheorem{theorem}{Theorem}
\newtheorem{lemma}[theorem]{Lemma}
\newtheorem{proposition}[theorem]{Proposition}

\theoremstyle{definition}

\newtheorem{question}[theorem]{Question}

\begin{document}

\title{The Lie algebra of type $G_2$ is rational over its quotient by the adjoint action}
\author{Dave Anderson}
\address{Instituto Nacional de Matem\'atica Pura e Aplicada, Rio de Janeiro, RJ 22460-320 Brasil}
\email{dave@impa.br}
\author{Mathieu Florence}
\address{Institut de Math\'ematiques de Jussieu, Universit\'e Paris 6,
place Jussieu, 75005 Paris, France}
\email{florence@math.jussieu.fr}
\author{Zinovy Reichstein}
\address{Department of Mathematics\\University of British Columbia\\ BC, Canada V6T 1Z2}
\email{reichst@math.ubc.ca}
\thanks{D.A. was partially supported by NSF Grant DMS-0902967.
Z.R. was partially supported by
National Sciences and Engineering Research Council of
Canada grant No. 250217-2012.}

\begin{abstract} 
Let $G$ be a split simple group of type $G_2$ over a field $k$, and let $\lieg$ be its Lie algebra.  Answering a question of J.-L.~Colliot-Th\'el\`ene, B.~Kunyavski\u\i, V.~L.~Popov, and Z.~Reichstein, we show that the function field $k(\lieg)$ is generated by algebraically independent elements over the field of adjoint invariants $k(\lieg)^G$.
\end{abstract}

\maketitle

\vspace{-.5in}

\begin{otherlanguage}{french}
\begin{abstract} 
Soit $G$ un groupe alg\'ebrique  simple et d\'eploy\'e de type $G_2$ sur un corps $k$. Soit $\lieg$ son alg\`ebre de Lie.  On d\'emontre que le corps des fonctions $k(\lieg)$ est transcendant pur sur le corps $k(\lieg)^G$ des invariants adjoints. Ceci r\'epond par l'affirmative \`a une question pos\'ee par J.-L.~Colliot-Th\'el\`ene, B.~Kunyavski\u\i, V.~L.~Popov et Z.~Reichstein.
\end{abstract}
\end{otherlanguage}

\medskip

{\bf I. Introduction.} Let $G$ be  a split connected reductive group 
over a field $k$ and let $\lieg$ be the Lie algebra of $G$. 
We will be interested in the following natural question:

\begin{question} \label{q1}
Is the function field $k(\lieg)$ \emph{purely transcendental} over the field of invariants $k(\lieg)^G$
for the adjoint action of $G$ on $\lieg$? That is, can $k(\lieg)$ be generated over $k(\lieg)^G$ 
by algebraically independent elements?
\end{question}

In \cite{trans}, the authors reduce this question to the case where $G$ is simple, and
show that in the case of simple groups, the answer is affirmative for split groups of types $A_n$ and $C_n$, and negative for all other types except possibly for $G_2$.  The standing assumption in \cite{trans} is that $\Char(k) = 0$, 
but here we work in arbitrary characteristic.

The purpose of this note is to settle Question~\ref{q1} for the remaining case $G=G_2$.

\begin{theorem}\label{t:main}
Let $k$ be an arbitrary field and $G$ be the simple split $k$-group 
of type $G_2$.  Then $k(\lieg)$ is purely transcendental over $k(\lieg)^G$.
\end{theorem}

\noindent
Under the same hypothesis, and also assuming $\Char(k)=0$, it follows from Theorem~\ref{t:main} and \cite[Theorem~4.10]{trans} that the field extension $k(G)/k(G)^G$ is also purely transcendental, where $G$ acts on itself by conjugation.

Apart from settling the last case left open in~\cite{trans}, we 
were motivated by the
(still mysterious) connection between Question~\ref{q1} and 
the Gelfand-Kirillov (GK) conjecture~\cite{gk}. In this context $\Char(k) = 0$.
A.~Premet~\cite{premet} recently showed that the GK conjecture fails for simple 
Lie algebras of any type other than $A_n$, $C_n$ and $G_2$. 
His paper relies on the negative results of~\cite{trans} and 
their characteristic $p$ analogues (\cite{premet}, 
see also \cite[Theorem 6.3]{trans}).
It is not known whether a positive answer 
to Question~\ref{q1} for $\lieg$ implies the GK conjecture 
for $\lieg$. The GK conjecture has been proved for algebras 
of type $A_n$ (see~\cite{gk}), but remains open for types 
$C_n$ and $G_2$.  While Theorem~\ref{t:main} does not settle 
the GK conjecture for type $G_2$, it puts the remaining two open 
cases---for algebras of type $C_n$ and $G_2$---on equal 
footing vis-\`a-vis Question~\ref{q1}.

\medskip
\noindent
{\bf II. Twisting.}  Temporarily, let $W$ be a linear algebraic group over 
a field $k$. (In the sequel, $W$ will be the Weyl group of $G$; in particular, 
it will be finite and smooth.)
We refer to \cite[Section 3]{dr}, \cite[Section 2]{florence}, 
or~\cite[Section 2]{trans} for details about the following facts.

Let $X$ be a quasi-projective variety with a (right) $W$-action defined over $k$, 
and let $\zeta$ be a (left) $W$-torsor over $k$.  The diagonal left action of $W$ on $X \times_{\Spec(k)} \zeta$ (by $g.(x,z)=(xg^{-1}, gz)$) makes 
$X \times_{\Spec(k)} \zeta$ into the total space of a $W$-torsor $X \times_{\Spec(k)} \zeta \to B$.  The base space $B$ of this torsor is usually called the 
{\em twist} of $X$ by $\zeta$.  We denote it by ${ }^\zeta X$. 

It is easy to see that if $\zeta$ is trivial then ${ }^\zeta X$ is 
$k$-isomorphic to $X$.  Hence, ${ }^\zeta X$ is 
a $k$-form of $X$, i.e., $X$ and ${ }^\zeta X$ become isomorphic 
over an algebraic closure of $k$.

The twisting construction is functorial in $X$: a $W$-equivariant morphism 
$X \to Y$ (or rational map $X\dasharrow Y$) induces a $k$-morphism ${ }^\zeta X \to { }^\zeta Y$  (resp., rational map ${ }^\zeta X \dasharrow { }^\zeta Y$).  

\medskip
\noindent
{\bf III. The split group of type $G_2$.}  We fix notation and briefly review the basic facts, referring to \cite{sv}, \cite{g2thesis}, or \cite{g2chern} for more details.  Over any field $k$, a simple split group $G$ of type $G_2$ has a faithful seven-dimensional representation $V$.  Following \cite[(3.11)]{g2chern}, one can fix a basis $f_1,\ldots,f_7$, with dual basis $X_1,\ldots,X_7$, so that $G$ preserves the nonsingular quadratic norm $N=X_1 X_7 + X_2 X_6 + X_3 X_5 + X_4^2$.  (See \cite[\S6.1]{g2thesis} for the case $\Char(k)=2$.  In this case $V$ is not irreducible, since the subspace spanned by $f_4$ is invariant; the quotient $V/(k\cdot f_4)$ is the minimal irreducible representation.  However, irreducibility will not be necessary in our context.)  
The corresponding embedding $G \hookrightarrow  \GL_7$ yields a split maximal torus and Borel subgroup $T \subset B \subset G$, by intersecting with diagonal and upper-triangular matrices.  Explicitly, the maximal torus is 
\begin{equation} \label{e.T}
T = \mathrm{diag}(t_1,\,t_2,\,t_1t_2^{-1},\,1,\,t_1^{-1}t_2,\,t_2^{-1},\,t_1^{-1}); 
\end{equation}
cf.~\cite[Lemma~3.13]{g2chern}.

The Weyl group $W = N(T)/T$ is isomorphic to the dihedral group with $12$ elements, and the surjection $N(T) \to W$ splits.  The inclusion $G \hookrightarrow \GL_7$ thus gives rise to an inclusion $N(T)= T \rtimes W \hookrightarrow D \rtimes \Sym_7$, where $D \subset \GL_7$ is the subgroup of diagonal matrices. On the level of the dual basis $X_1,\ldots,X_7$, we obtain an isomorphism $W\isom \Sym_3\times \Sym_2$ realized as follows:
$\Sym_3$ permutes the three ordered pairs $(X_1,X_7)$, $(X_2,X_6)$ and $(X_3,X_5)$, 
and $\Sym_2$ exchanges the two ordered triples $(X_1,X_5,X_6)$ and $(X_7,X_3,X_2)$.  The variable $X_4$ is fixed by $W$.  For details, see \cite[\S A.3]{g2chern}.

The subgroup $P\subset G$ stabilizing the isotropic line spanned by $f_1$ is a maximal standard parabolic, and the corresponding homogeneous space $P\backslash G$ is isomorphic to the five-dimensional quadric $\Qq \subset \PP(V)$ defined by the vanishing of the norm, i.e., by the equation
\begin{equation}\label{e.Q}
 X_1 X_7 + X_2 X_6 + X_3 X_5 + X_4^2=0.
\end{equation}
Note that the quadric  $\Qq$ is endowed with an action of $T$. 
An easy tangent space computation shows that 
$P$ is smooth regardless of the characteristic of $k$.

\begin{lemma}\label{l:special}
The group $P$ is {\em special}, i.e., $H^1(l, P) = \{ 1 \}$ for every field extension $l/k$. Moreover, $P$ is rational, as a variety over $k$.
\end{lemma}

\begin{proof}
Since the split group of type $G_2$ is defined over the prime field, we may replace $k$ by the prime field for the purpose of proving this lemma, and in particular, we can assume $k$ is perfect. We begin by briefly recalling a construction of Chevalley \cite{chevalley}.  The isotropic line $E_1\subset V$ stabilized by $P$ is spanned by $f_1$, and $P$ also preserves an isotropic $3$-space $E_3$ spanned by $f_1,f_2,f_3$; see, e.g., \cite[\S2.2]{g2chern}.  There is a corresponding map $P \to \GL(E_3/E_1) \isom \GL_2$, which is a split surjection thanks to the block matrix described in \cite[p.~13]{heinloth} as the image of ``$B$'' in $\GL_7$.  
The kernel is unipotent, and we have a split exact sequence corresponding to the Levi decomposition:
\begin{equation} \label{e.parabolic}
  1 \to R_u(P) \to P \to \GL_2 \to 1.
\end{equation}
Combining the  exact sequence in cohomology induced by~\eqref{e.parabolic}  
with the fact that both $R_u(P)$ and $\GL_2$ are special 
(see \cite[pp. 122 and 128]{serre}), shows that $P$ is special.

Since $P$ is isomorphic to $R_u(P) \times \GL_2$ as a variety over $k$,
and $P$ is smooth, so is $R_u(P)$.  A smooth connected unipotent group over 
a perfect field is rational~\cite[IV, \S2(3.10)]{dg}; therefore $R_u(P)$ is $k$-rational, and so is $P$.
\end{proof}

\medskip
\noindent
{\bf IV. Proof of Theorem~\ref{t:main}.} 
Keep the notation of the previous section.  
By a {\it $W$-model} (of $k(\Qq)^T$), we mean a quasi-projective $k$-variety $Y$, endowed with a right action of $W$, together with a dominant $W$-equivariant $k$-rational map $\Qq \dasharrow Y$ which, on the level of function fields, identifies $k(Y)$ with $k(\Qq)^T$.  Such a map $\Qq \dasharrow Y$ is called a {\it ($W$-equivariant) rational quotient map}.  
A $W$-model is unique up to a $W$-equivariant 
birational isomorphism; we will construct an explicit one below. 

We reduce Theorem~\ref{t:main} to a statement about rationality of a twisted $W$-model, in two steps.  The first holds for general split connected semisimple groups $G$.

\begin{proposition} \label{prop.red1} 
Let $G$ be a split connected semisimple group over $k$, with split 
maximal $k$-torus $T$.  Let $K=k(\liet)^W$, $L=k(\liet)$, and let $\zeta$ be the $W$-torsor corresponding to the field extension $L/K$.  If the twisted variety ${ }^{\zeta} (G_K/T_K)$ 
is rational over $K$, 
then $k(\lieg)$ is purely transcendental over $k(\lieg)^G$.
\end{proposition}

\begin{proof} 
Consider the $(G \times W)$-equivariant morphism
\[ f \colon  G/T \times_{\Spec(k)} \liet \to \lieg \] 
given by $(\overline{a}, t) \mapsto \mathrm{Ad}(a)t$, 
where $\liet$ is the Lie algebra of $T$,
$\overline{a} \in G/T$ is the class of $a \in G$, modulo $T$. 
Here $G$ acts on $G/T \times \liet$ by translations 
on the first factor (and trivially on $\liet$), and via the adjoint 
representation on $\lieg$.
The Weyl group $W$ naturally acts on $\liet$ and $G/T$ (on 
the right), diagonally on $G/T \times \liet$, and
trivially on $\lieg$.  

The image of $f$ contains the semisimple locus in $\lieg$, so 
$f$ is dominant and induces an inclusion $f^* \colon k(\lieg) \hookrightarrow k(G/T \times \liet)$. 
Clearly $f^* \, k(\lieg) \subset k(G/T \times \liet)^W$. 
We will show that in fact
\begin{equation} \label{e.claim}
f^* \, k(\lieg) = k(G/T \times \liet)^W \, .  
\end{equation}
Write $\overline{k}$ for an algebraic closure of $k$, 
and note that the preimage of a $\bar{k}$-point
of $\lieg$ in general position is a single $W$-orbit in
$(G/T \times \liet)_{\overline{k}}$. 
To establish \eqref{e.claim}, it remains to check that $f$ is smooth at a general point $(g, x)$ 
of $G/T \times \liet$.  (In particular, when $\Char(k)=0$ nothing more is needed.)  
To carry out this calculation, 
we may assume without loss of generality that $k$ is algebraically 
closed and (since $f$ is $G$-equivariant) $g = 1$.  Since 
$\dim(G/T \times \liet) = \dim(\lieg)$, it suffices to show
that the differential 
\[ 
 df \colon T_{(1, x)}(G/T \times \liet) \to T_x(\lieg) 
\]
is surjective, for any regular semisimple element $x\in\liet$.  
Equivalently, we want to show that
$[x, \lieg] + \liet = \lieg$.  Since $x$ is regular, we have 
$\dim([x,\lieg])+\dim\liet = \dim\lieg$. Thus it remains to show that  
$[x,\lieg] \cap \liet = 0$.  To see this, suppose $[x,y]\in \liet$ 
for some $y\in\lieg$. Since $x$ is semisimple, we can write
$y = \sum_{i = 1}^r y_{\lambda_i}$, where $y_{\lambda}$ is an
eigenvector for $\mathrm{ad}(x)$ with eigenvalue $\lambda$,
and $\lambda_1, \dots, \lambda_r$ are distinct. 
Then $[x, y] = \sum_{i = 1}^r  \lambda_i y_{\lambda_i} \in \liet$ 
is an eigenvector for $\ad(x)$ with eigenvalue $0$.  
Remembering that eigenvectors of $\ad(x)$ with distinct eigenvalues are
linearly independent, we conclude that $[x, y] = 0$. This
completes the proof of \eqref{e.claim}.

It is easy to see $k(G/T \times \liet)^{G \times W} = k(\liet)^W$.  Summarizing, $f^*$ induces a diagram
 \[ \xymatrix{k(G/T \times_{\Spec(k)} \liet)^W  & \ar@{-}[l]_<<<<<{\sim} k(\lieg) 
\\
 k(\liet)^W   \ar@{-}[u] & \ar@{-}[l]_{\sim} k(\lieg)^G , \ar@{-}[u]} \]
where the top row is the $G$-equivariant isomorphism~\eqref{e.claim}, 
and the bottom row is obtained from the top by taking $G$-invariants.  
Note that 
\[
 k(G/T \times_{\Spec(k)} \liet) \simeq 
K((G/T)_K \times_{\Spec(K)} \Spec{L}), 
\]
where $\simeq$ denotes a $G$-equivariant isomorphism of fields.
(Recall that $G$ acts trivially on $\liet$ and hence also on $L$ and $K$.)
Thus the field extension on the left side of our diagram
can be rewritten as $K(^\zeta (G_K/T_K))/K$, 
where $\zeta$ is the $W$-torsor $\Spec(L) \to \Spec(K)$.  By assumption, 
this field extension is purely transcendental; the diagram shows it is isomorphic to $k(\lieg)/k(\lieg)^G$.
\end{proof}

For the second reduction, we return to the assumptions of Section III.
\begin{proposition}\label{prop.red1b}
Let $G$ be a split simple group of type $G_2$, with maximal torus $T$ and Weyl group $W$, and let $\Qq$ be the quadric defined in Section III.  Suppose that 
for a given $W$-model $Y$ of $k(\Qq)^T$, and for some $W$-torsor $\zeta$ over some field $K/k$, the twisted variety ${ }^{\zeta} (Y_K)$ is rational over $K$.  Then the twisted variety ${}^{\zeta}(G_K/T_K)$ is rational over $K$.
\end{proposition}

\begin{proof}
For the purpose of this proof, we may view $K$ as a new base field and replace it with $k$. 

We claim that the left action of $P$ on $G/T$ is generically free.  
Since $G$ has trivial center, the (characteristic-free) argument at the beginning of the proof of~\cite[Lemma 9.1]{trans} shows that in order to establish this 
claim it suffices to show that the right $T$-action on  
$\Qq = P \backslash G$ is generically free. The latter action, 
given by restricting the linear action~\eqref{e.T}
of $T$ on $\PP^6$ to the quadric $\Qq$ given by~\eqref{e.Q},  
is clearly generically free.

Let $Y$ be a $W$-model. The $W$-equivariant rational map 
$G/T  \dasharrow Y$ induced by the projection 
$G \to P \backslash G = \Qq$
is a rational quotient map for the left $P$-action on $G/T$;
cf.~\cite[p. 458]{trans}. Since the $P$-action is generically free, this map is 
a $P$-torsor over the generic point of $Y$; see~\cite[Theorem 4.7]{bf}.
By the functoriality of the twisting operation, after twisting
by a $W$-torsor $\zeta$, we obtain a rational map 
${ }^\zeta (G/T)  \dasharrow { }^\zeta Y$, which is a
$P$-torsor over the generic point of
${ }^\zeta Y$. This torsor has a rational section, since $P$ is 
special; see Lemma~\ref{l:special}.  In particular,
${ }^\zeta (G/T)$ is $k$-birationally isomorphic to $P \times  { }^\zeta Y$.  
Since $P$ is $k$-rational (once again, by Lemma~\ref{l:special}), 
${ }^\zeta (G/T)$ is rational over ${ }^\zeta Y$. Since ${ }^\zeta Y$ is rational over $k$, we conclude that so is
${ }^\zeta (G/T)$, as desired. 
\end{proof}

It remains to show that the hypothesis of Proposition~\ref{prop.red1b} holds.  
As before, we may replace the field $K$ with $k$. The following lemma completes the proof of Theorem~\ref{t:main}.

\begin{lemma} \label{lem.red2} 
Let $Y$ be a $W$-model. The twisted variety ${ }^{\zeta} Y$ is rational over $k$, 
for every $W$-torsor $\zeta$ over $k$.
\end{lemma}

\begin{proof}
We begin by constructing an explicit $W$-model.  The affine open subset
$\Qq^{\rm aff} = \{x_1 x_7 + x_2 x_6 + x_3 x_5 + 1 =0 \} \subset \AA^6$ 
(where $X_4\neq0$) is $N(T)$-invariant. 
Here the affine coordinates on $\AA^6$ are $x_i := X_i/X_4$, for $i \neq 4$. 
The field of rational functions invariant for the $T$-action 
on $\Qq^{\rm aff}$ is $k(y_1, y_2, y_3,z_1,z_2)$, where the variables
\begin{align*}
 y_1 = x_1 x_7,\quad y_2 = x_2 x_6,\quad y_3 = x_3 x_5,\quad z_1 = x_1 x_5 x_6,\;\text{ and }\; z_2 = x_2 x_3 x_7
\end{align*}
are subject to the relations 
$y_1 + y_2 + y_3 + 1=0$ and $y_1 y_2 y_3 = z_1 z_2$.  
Thus we may choose as a $W$-model the affine subvariety $\Lambda_1$ of $\AA^5$
given by these two equations, where $W = \Sym_2 \times \Sym_3$ acts on the coordinates as follows: 
$\Sym_2$ permutes $z_1, z_2$, and $\Sym_3$ permutes $y_1, y_2, y_3$.  (Recall the $W$-action defined in Section III, and note that the field $k(\Qq)$ is recovered by adjoining the classes of variables $x_1$ and $x_2$.)  
We claim that $\Lambda_1$ is $W$-equivariantly birationally isomorphic to
\[ \begin{array}{l} 
\Lambda_2 = \{ (Y_1:Y_2:Y_3: Z_0: Z_1: Z_2) \, : \, \text{$Y_1 + Y_2 + Y_3 +Z_0=0$ and $Y_1 Y_2 Y_3 = Z_1 Z_2 Z_0$} \} \subset \PP^5, \\
\Lambda_3 = \{ (Y_1:Y_2:Y_3: Z_1: Z_2) \, : \,  Y_1 Y_2 Y_3 + (Y_1 + Y_2 + Y_3) Z_1 Z_2=0 \} \subset \PP^4, \; \text{and} \\
\Lambda_4 = \{ (Y_1:Y_2:Y_3: Z_1: Z_2) \, : \, Z_1 Z_2 +Y_2 Y_3 + Y_1 Y_3 + Y_1 Y_2=0  \} \subset \PP^4 \, ,
\end{array} \]
where $W$ acts on the projective coordinates $Y_1, Y_2, Y_3, Z_1, Z_2, Z_0$ as follows: 
$\Sym_2$ permutes $Z_1, Z_2$, $\Sym_3$ permutes $Y_1, Y_2, Y_3$, 
and every element of $W$ fixes $Z_0$.
Note that $\Lambda_2 \subset \PP^5$ is the projective closure of $\Lambda_1 \subset \AA^5$; hence, using $\simeq$ to denote $W$-equivariant birational equivalence, we have $\Lambda_1 \simeq \Lambda_2$.
The isomorphism $\Lambda_2 \simeq \Lambda_3$ is obtained by 
eliminating $Z_0$ from the 
system of equations defining $\Lambda_2$.
Finally, the isomorphism $\Lambda_3 \simeq \Lambda_4$ comes from the Cremona transformation $\PP^4 \dasharrow \PP^4$ given by 
$Y_i \to 1/Y_i$ and $Z_j \to 1/Z_j$ for $i = 1, 2, 3$ and $j = 1, 2$.

Let $\zeta$ be a $W$-torsor over $k$. It remains to be shown that  ${ }^\zeta \Lambda_4$ 
is $k$-rational.
Since $\Lambda_4$ is a $W$-equivariant quadric 
hypersurface in $\PP^4$, and the $W$-action on $\PP^4$ 
is induced by a linear representation $W \to \GL_5$,
Hilbert's Theorem 90 tells us that ${ }^\zeta \PP^4$ is $k$-isomorphic to $\PP^4$, 
and ${ }^\zeta \Lambda_4$ is isomorphic to a quadric hypersurface in $\PP^4$
defined over $k$; see~\cite[Lemma 10.1]{dr}. It is easily checked 
that $\Lambda_4$ is smooth over $k$, and therefore so is ${ }^\zeta \Lambda_4$. The zero-cycle of degree 3 given by   
$(1:0:0:0:0)+(0:1:0:0:0)+(0:0:1:0:0)$ in $\Lambda_4$ is $W$-invariant, 
so it defines a zero-cycle of degree 3 in ${ }^\zeta \Lambda_4$. 
By Springer's theorem, the smooth quadric ${ }^\zeta \Lambda_4$ has 
a $k$-rational point, hence is $k$-rational. 
\end{proof}

\noindent
{\bf Acknowledgement.} 
We are grateful to J.-L.~Colliot-Th\'el\`ene for stimulating conversations.



\end{document}